\documentclass[sn-mathphys-num]{sn-jnl}

\usepackage{graphicx}%
\usepackage{multirow}%
\usepackage{amsmath,amssymb,amsfonts}%
\usepackage{amsthm}%
\usepackage{mathrsfs}%
\usepackage[title]{appendix}%
\usepackage{xcolor}%
\usepackage{textcomp}%
\usepackage{manyfoot}%
\usepackage{booktabs}%
\usepackage{algorithm}%
\usepackage{algorithmicx}%
\usepackage{algpseudocode}%
\usepackage{listings}%

\theoremstyle{thmstyleone}%
\newtheorem{theorem}{Theorem}
\newtheorem{cor}[theorem]{Corollary}%

\theoremstyle{thmstyletwo}%
\newtheorem{example}{Example}%
\newtheorem{remark}{Remark}%

\theoremstyle{thmstylethree}%

\def\x{\mathbf{x}}

\def\K{\mathbf{K}}
\def\B{\mathbf{B}}
\def\M{\mathbf{M}}
\def\B{\mathbf{B}}
\def\X{\mathbf{X}}
\def\P{\mathbf{P}}
\def\K{\mathbf{K}}
\def\R{\mathbb{R}}
\def\N{\mathbb{N}}

\def\P{\mathbf{P}}

\def\K{\mathbf{K}}
\def\B{\mathbf{B}}
\def\Q{\mathbf{Q}}
\def\M{\mathbf{M}}

\def\X{\mathbf{X}}

\def\v{\mathbf{v}}
\def\f{\mathbf{f}}

\def\f{\mathbf{f}}
\def\w{\mathbf{w}}

\def\x{\mathbf{x}}

\def\N{\mathbb{N}}

\def\om{\mathbf{\Omega}}

\def\bphi{\boldsymbol{\phi}}

\def\bmu{\boldsymbol{\mu}}
\def\balpha{\boldsymbol{\alpha}}
\def\bbeta{\boldsymbol{\beta}}
\def\bgamma{\boldsymbol{\gamma}}

\def\B{\mathbf{B}}

\def\C{\mathbf{C}}
\def\X{\mathbf{X}}

\def\supmu{{\rm supp}(\phi)}

\def\bphi{\boldsymbol{\phi}}

\begin{document}

\title[Rank conditions for exactness of SDP-relaxations]{Rank conditions for exactness of
semidefinite relaxations  in polynomial optimization}


\author*[1,2]{\fnm{Jean B. }\sur{Lasserre}}\email{lasserre@laas.fr}

%

\affil*[1]{\orgdiv{LAAS-CNRS}, \orgaddress{\street{7 Av. Colonel Roche}, \city{Toulouse}, \postcode{31031},  \country{France}}}

\affil[2]{\orgdiv{Dept. Math}, \orgname{Toulouse School of Economics}, \orgaddress{\city{Toulouse}, \postcode{31000}, \country{France}}}


\abstract{The abstract serves both as a general introduction to the topic and as a brief, non-technical summary of the main results and their implications. Authors are advised to check the author instructions for the journal they are submitting to for word limits and if structural elements like subheadings, citations, or equations are permitted.}

\abstract{We consider the Moment-SOS hierarchy in polynomial optimization.
We first provide a sufficient condition to solve the truncated $\K$-moment problem
associated with a given degree-$2n$ pseudo-moment sequence $\bphi^n$ and a semi-algebraic set $\K\subset\R^d$. 
Namely, let $2v$ be the maximum degree of the polynomials that describe $\K$. If
the rank $r$ of its associated moment matrix is less than $n-v+1$, then 
$\bphi^n$ restricted to degree-$2n-1$ moments,
has an atomic representing measure supported on at most $r$ points of $\K$.
When used at step-$n$ of the Moment-SOS hierarchy, it provides a sufficient condition 
to guarantee its finite convergence (i.e., the optimal value of the corresponding 
degree-$n$ semidefinite relaxation of the hierarchy is the global minimum). 
For Quadratic Constrained Quadratic Problems (QCQPs) one may also recover global minimizers from
the optimal pseudo-moment sequence.
Our condition is in the spirit of Blekherman's rank condition and while on the one-hand it is more restrictive, on the other hand it applies to constrained POPs as it provides a \emph{localization} on $\K$ for the representing measure.}




\maketitle

\section{Introduction}
\label{sec:intro}

Consider the {\em polynomial} optimization problem (POP):
\begin{eqnarray}
\label{def-pb-1}
\P:\quad f^*&=&\min_\x\,\{\,f(\x):\:\x\,\in\,\K\,\}\\
\label{set-K}
\mbox{with}\quad\K&:=&\{\,\x\in\R^n:\: g_j(\x)\,\geq\,0,\quad j=1,\ldots,m\,\}\,,
\end{eqnarray}
where $f,g_j\in\R[\x]$ are polynomials. 
In particular, if $\mathrm{deg}(f)\leq 2$ and $\mathrm{deg}(g_j)\leq 2$ for all 
$j=1,\ldots,m$, then $\P$ is called a quadratic constrained quadratic problem (QCQP). This latter class
contains many important problems for which computing (or even approximating) $f^*$  ``efficiently"
is a scientific challenge; indeed $\P$ is NP-hard in general.

A popular strategy to compute (or approximate) $f^*$ is to provide a monotone non decreasing sequence
of lower bounds that converges to $f^*$ from below.  Some LP and semidefinite (SDP) relaxations
introduced in the nineties \cite{Sherali} and the 2000's \cite{lass-cras,lass-siopt-2001} provide an example of such a strategy and the interested reader 
is referred to e.g. \cite{lass-mor,Laurent} for an analysis of their respective advantages and drawbacks. 
 
This paper is concerned with  the {\em Moment-SOS hierarchy} \cite{book-hierarchy,lass-icm}
which applies to solve not only POPs but also many important problems 
in Science \& Engineering, provided that they 
are modeled as instances of the Generalized Moment Problem (GMP) with algebraic data;
the interested reader is also referred to e.g. \cite{lass-acta} for a recent exposition of such applications.

The Moment-SOS hierarchy for solving $\P$ 
consists of a nested sequence of semidefinite relaxations $(\Q_n)_{n\in\N}$ of $\P$,
whose size increases with $n$ and whose associated sequence of optimal values  $(\rho_n)_{n\in\N}$ 
is monotone non decreasing and converges to $f^*$ as $n$ increases.  Moreover
it has been shown \cite{Nie-1,Nie-2} that generically (i) its convergence is finite, i.e., $\rho_n=f^*$
at some step-$n$ of the hierarchy, and (ii) 
extraction of global minimizers can be done 
by exploiting a flatness condition due to Curto \& Fialkow \cite{Curto} (and related
to a certain rank condition on moment matrices). When $\rho_n=f^*$ the semidefinite relaxation
$\Q_n$ is said to be \emph{exact}.

\paragraph{On rank conditions}Let $\bphi=(\phi_{\balpha})_{\balpha\in\N^d_{2n}}$, be a 
real sequence (up to degree $2n$) with 
positive semidefinite (psd) moment matrix $\M_n(\bphi)\succeq0$ (see definition in \S \ref{notation}).  Then
$\bphi$ has a representing measure if there exists a measure $\phi$ on $\R^d$ such that
$\phi_{\balpha}=\int \x^{\balpha}\,d\phi$ for all $\balpha\in\N^d_{2n}$.
To identify whether a sequence $\bphi$ has a representing measure,
an important result  is the (unconstrained) flatness condition of Curto and Fialkow \cite{Curto,Fialkow} which states that if
\begin{equation}
\label{eq:flat-1}
 \mathrm{rank}(\M_n(\bphi))\,=\,\mathrm{rank}(\M_{n-1}(\bphi))\,,
\end{equation}
then $\bphi$ has an atomic representing measure on $\R^d$ supported on $\mathrm{rank}(\M_n(\bphi))$ atoms.
Similarly with $d_j:=\lceil\mathrm{deg}(g_j)/2\rceil$, and
$g_j(\x)=\sum_{\bbeta}g_{j,\bbeta}\,\x^{\bbeta}$, $j=1,\ldots,m$,
let $g_j\,\bphi=(g_j\bphi)_{\balpha}$, $\balpha\in\N^d_{2n}$, be the sequence where
$(g_j\,\phi)_{\balpha}=\sum_{\bbeta}g_{j,\bbeta}\,\phi_{\balpha+\bbeta}$ for all $\balpha\in\N^d_{2n}$.
Then suppose that 
$\bphi$ also satisfies $\M_{n-d_j}(g_j\,\bphi)\succeq0$, $j=1,\ldots,m$, and let  $v:=\max_j d_j$.
The constrained flatness condition of Curto and Fialkow \cite{Curto,Fialkow} states that if
\begin{equation}
\label{eq:flat-2}
 \mathrm{rank}(\M_n(\bphi))\,=\,\mathrm{rank}(\M_{n-v}(\bphi))\,,
 \end{equation}
then $\bphi$ has an atomic representing measure supported on $\mathrm{rank}(\M_{n}(\bphi))$
atoms in $\K$.
Finally,  by a result of Blekherman \cite{Blek}, it turns out that if
\begin{equation}
\label{eq:flat-3}
 \mathrm{rank}(\M_n(\bphi))\,\leq\,\left\{\begin{array}{rl}3n-3&\mbox{if $n\geq 3$}\\
 5&\mbox{if $n=2$,}\end{array}\right.\\
\end{equation}
then the subsequence $\bphi^{\hat{n}}$ ($:= (\phi_{\balpha})_{\balpha\in\N^d_{2n-1}}$) 
of $\bphi$ has a representing measure on $\R^d$; see \cite{Fialkow}\footnote{In \cite[Theorem 2.36]{lass-book}
Blekherman's result \eqref{eq:flat-3} is incorrectly stated. Indeed  
only the subsequence $\bphi^{\hat{n}}$ of moments up to degree $2n-1$ has a representing measure,
and not the whole sequence $\bphi^n$ of moments up to degree $2n$ in general.}.

\begin{remark}
Importantly, notice that 
in contrast to \eqref{eq:flat-1}, on the one hand  the condition \eqref{eq:flat-3} is only concerned with
the single moment matrix $\M_n(\bphi)$, but on the other hand  there is no localization of the 
support of its measure. Moreover only the subsequence $\bphi^{\hat{n}}$
of $\bphi$ (and not $\bphi$) has a representing measure.
\end{remark}

\paragraph{Contribution}
We are concerned with practical sufficient rank-conditions for finite convergence 
of the Moment-SOS hierarchy. 
With $g_0(\x)=1$ for all $\x$, the degree-$n$ semidefinite relaxation $\Q_n$ of the Moment-SOS 
hierarchy associated with $\P$, reads:
\begin{equation}
\label{eq:step-n}
\Q_n:\quad\rho_n\,=\,\min_{\bphi}\,\{\,\phi(f):\: \phi(1)=1\,;\: \M_{n-d_j}(g_j\,\bphi)\succeq0\,,\:j=0,\ldots,m\,\}\,,\end{equation}
where $\bphi=(\phi_{\balpha})_{\balpha\in\N^d_{2n}}$, and
$\M_{d-d_j}(g_j\,\bphi)$ is the localizing matrix associated with $\bphi$
and the polynomial
$g_j\in\R[\x]$. ($\M_{n-d_0}(g_0\,\bphi)=\M_n(\bphi)$ is the moment matrix associated with $\bphi$.) 
With $\bphi$ (in bold) is associated the Riesz linear functional $\phi\in\R[\x]_{2n}^*$ (not in bold face) defined by:
\[p\:(=\sum_{\balpha\in\N^d_{2n}}p_{\balpha}\,\x^{\balpha})\quad\mapsto \phi(p)\,=\,\sum_{\balpha\in\N^d_{2n}} p_{\balpha}\,\phi_{\balpha}\,,\quad\forall p\in\R[\x]_{2n}\,.\]
The truncated $\K$-moment problem is concerned with conditions on $\bphi$ to guarantee
that $\bphi$ has a representing measure on $\K$, i.e., $\phi(p)=\int_\K pd\phi$ for all $p\in\R[\x]_{2n}$,
for some measure $\phi$ on $\K$; see e.g. \cite{Schmudgen,lass-book}.

Let $d_f:=\lceil\mathrm{deg}(f)/2\rceil$ and $d_j:=\lceil\mathrm{deg}(g_j)/2\rceil$, for all $j=1,\ldots,m$,
and $v:=\max\{d_1,\ldots,d_m\}$.
Then our first result provides a sufficient condition to solve the truncated $\K$-moment problem.

\begin{theorem}
\label{intro-th1}
With $n\geq v$, let $\bphi^n=(\phi_{\balpha})_{\balpha\in\N^d_{2n}}$ be such that 
\begin{equation}
\phi(1)\,=\,1\,;\quad \M_n(\bphi^n)\,\succeq0\,;\quad \M_{n-d_j}(g_j\,\bphi^n)\,\succeq0\,,\quad j=1,\ldots,m\,.
\end{equation}

If $s:=\mathrm{rank}(\M_n(\bphi^n))\,\leq\,n-v+1$ 
then $\bphi^{\hat{n}}:=(\phi^n_{\balpha})_{\balpha\in\N^d_{2n-1}}$ 
has a representing measure supported on 
$r\,(\leq s)$ points of $\K$. In addition, if $r=s$ then
the whole sequence $\bphi^{n}=(\phi^n_{\balpha})_{\balpha\in\N^d_{2n}}$ has a representing measure supported on $s$ points of $\K$.
\end{theorem}

Our second result investigates the impact of Theorem \ref{intro-th1} on the Moment-SOS hierarchy.
\begin{theorem}
\label{intro-th-main}
Let $\P$ be as in \eqref{def-pb-1}.
For every $n$ with $n\geq v$ and $2n-1\geq \mathrm{deg}(f)$, let 
$\bphi^n=(\phi^n_{\balpha})_{\balpha\in\N^d_{2n}}$ 
be  an optimal solution of the semidefinite relaxation $\Q_n$:

(i) If ${\rm rank}(\M_n(\bphi^n))\leq n-v+1$ then
$\phi^n(f)=f^*$, that is, $\rho_n=f^*$ and the relaxation $\Q_n$ is exact.

(ii) Next assume that $\mathrm{deg}(f)\leq 2$ and $\mathrm{deg}(g_j\leq 2)$ for all $j$,
so that $v=1$ (and $\P$ is a QCQP). Then one may recover  a probability measure
supported on global minimizers of $\P$ (e.g., via the extraction procedure of \cite{extract}).
\end{theorem}
\vspace{.1cm}

\begin{remark}
 It is well-known that if $\mathrm{rank}(\M_n(\bphi^n))=1$
then  $\bphi^n$ is the moment vector of the Dirac measure $\delta_{\x^*}$ for some global minimizer $\x^*\in\K$.
So it is fair to say that the rank-condition in Theorem
\ref{intro-th1}  and Theorem \ref{intro-th-main} provides an extension  of this result.

This rank condition is in the spirit of Blekherman's condition \eqref{eq:flat-3}
(i.e., with no flatness condition as in \eqref{eq:flat-1}).
While on the one hand it is more restrictive than \eqref{eq:flat-3}, on the other hand it
provides an additional localization on $\K$ of the support of representing measure. 
This localization feature is crucial for polynomial optimization 
as described in Theorem \ref{intro-th-main}.
\end{remark}
\vspace{.2cm}

In fact, in \cite[p. 72]{Blek} {\color{blue}the statement that \eqref{eq:flat-3} provides a stopping criterion for exactness
of  the hierarchy of SOS relaxations 
associated with polynomial optimization problems (POPs) is misleading the reader (see the introduction of section \S \ref{sec:exactness} for a detailed
comment). This is because  
\eqref{eq:flat-3} is a rank condition on the moment matrix only, with \emph{no} condition on 
the localization constraint,
 and therefore the representing measure has no reason to be supported on the feasible set 
 of the optimization problem. It would be a striking (and surprising) result if in Theorem \ref{intro-th-main}(i)
 one might replace $n-v+1$ by $3n-3$ (as in \eqref{eq:flat-3}). }
On the other hand, 
in Section \ref{sec:exactness} we prove that indeed Blekherman's condition \eqref{eq:flat-3} provides a sufficient condition 
to detect whether the \emph{single} semidefinite relaxation associated with the \emph{unconstrained}
POP: $\inf \{f(\x): \x\in\R^d\}$ (where $f$ is an even degree polynomial) is exact. The proof is 
not trivial because  if \eqref{eq:flat-3} holds (i.e., $\mathrm{rank}(\M_n(\bphi))\leq 3n-3$) then only the subsequence
of moments $(\phi_{\balpha})_{\balpha\in\N^d_{2n-1}}$ up to degree $2n-1$ {\color{blue}is guaranteed to have} a representing measure, say $\mu$, on $\R^d$.
{\color{blue}As $f$
is necessarily of degree $2n$ (and not of degree $2n-1$), $\bphi(f)$ may not be equal to
$\int fd\mu$.}

\section{Notation, definitions and preliminary results}
\label{notation}
Let $\mathbb{R}[\x]$ denote the ring of polynomials in the variables $\x=(x_1,\ldots,x_d)$ and let
$\mathbb{R}[\x]_n$ be the vector space of polynomials of degree at most $n$
(whose dimension is $s(n):={n+d\choose n}$).
For every $n\in\N$, let  $\N^d_n:=\{\alpha\in\N^d:\vert\alpha\vert \,(=\sum_i\alpha_i)\leq n\}$,
and
let $\v_n(\x)=(\x^{\balpha})$, $\balpha\in\N^d$, be the vector of monomials of the canonical basis
$(\x^{\balpha})$ of $\mathbb{R}[\x]_{n}$.
A polynomial $f\in\mathbb{R}[\x]_n$ is written
\[\x\mapsto f(\x)\,=\,\sum_{\balpha\in\N^d_n}f_{\balpha}\,\x^{\balpha}\,=\,\langle\f,\v_n(\x)\rangle\,,\]
where $\f=(f_{\balpha})\in\mathbb{R}^{s(n)}$ is its vector of coefficients in the canonical basis of monomials $(\x^{\balpha})_{\balpha\in\N^d}$.
For real symmetric matrices, let $\langle \B,\C\rangle:={\rm trace}\,(\B\C)$ while the notation $\B\succeq0$
stands for $\B$ is positive semidefinite (psd) whereas
$\B\succ0$ stands for $\B$ is  positive definite (pd).
\vspace{.2cm}

\noindent
{\bf The Riesz linear functional.}
Given a sequence $\bphi=(\phi_\alpha)_{\balpha\in\N^d}$ (with $\bphi$ in bold), the Riesz functional is the linear mapping
$\phi:\mathbb{R}[\x]\to\mathbb{R}$ (with $\phi$ not in bold) defined by:
\begin{equation}
\label{Riesz}
f\:(=\sum_{\balpha} f_{\balpha}\,\x^{\balpha})\quad \mapsto \phi(f)\,=\,\sum_{\balpha\in\N^d}f_{\balpha}\,\phi_{\balpha}\,,
\quad\forall f\in\R[\x]\,.
\end{equation}
\noindent
{\bf Moment matrix.}
The degree-$n$ {\it moment} matrix associated with a sequence
$\bphi=(\phi_{\balpha})$, $\balpha\in\N^d$, is the real symmetric matrix $\M_n(\bphi)$ with rows and columns indexed by $\N^d_n$, and whose entry
$(\balpha,\bbeta)$ is just $\phi_{\balpha+\bbeta}$, for every $\balpha,\bbeta\in\N^d_n$.
Alternatively, let
$\v_n(\x)\in\mathbb{R}^{s(n)}$ be the vector $(\x^{\balpha})$, $\balpha\in\N^d_n$, and
define the real symmetric matrices $(\B^1_{\balpha})$ by
\begin{equation}
\label{balpha}
\v_n(\x)\,\v_n(\x)^T\,=\,\sum_{\balpha\in\N^d_{2n}}\B^1_{\balpha}\,\x^{\balpha},\qquad\forall\x\in\mathbb{R}^d.\end{equation}
Then $\M_n(\bphi)=\sum_{\balpha\in\N^d_{2n}}\phi_{\balpha}\,\B^1_{\balpha}$.
If $\bphi$ has a representing measure $\phi$ then
$\M_n(\bphi)\succeq0$ because $\langle\f,\M_n(\bphi)\f\rangle=\int f^2d\phi\geq0$, for all $f\in\mathbb{R}[\x]_n$; in this case $\phi_{\balpha}$ is the $\balpha$-moment of $\phi$.

A measure {\color{blue}all of whose} moments are finite, is {\it moment determinate} if there is no other measure with same moments.
The support of a Borel measure $\phi$ on $\mathbb{R}^d$ (denoted $\supmu$) is the smallest closed set $\om$ such that $\phi(\mathbb{R}^d\setminus\om)=0$.

In the Theoretical Computer Science community, a vector $\bphi=(\phi_{\balpha})$ whose moment matrix is psd, is called a vector 
of {\em pseudo-moments} (and {\em moments} if $\bphi$ has a representing measure).
\vspace{.2cm}

\noindent
{\bf Localizing matrix.} With $\bphi$ as above and $g\in\mathbb{R}[\x]$ (with $g(\x)=\sum_{\bgamma} g_{\bgamma}\x^{\bgamma}$), the degree-$n$ {\it localizing} matrix associated with $\bphi$
and $g$ is the real symmetric matrix $\M_n(g\,\bphi)$ with rows and columns indexed by $\N^d_n$, and whose entry $(\balpha,\bbeta)$ is just $\sum_{\bgamma}g_{\bgamma}\phi_{\balpha+\bbeta+\bgamma}$, for every $\balpha,\bbeta\in\N^d_n$.
Alternatively, let $\B^g_{\balpha}$ be the real symmetric matrices defined by:
\begin{equation}
\label{calpha}
g(\x)\,\v_n(\x)\,\v_n(\x)^T\,=\,\sum_{\balpha\in\N^d_{2n+{\rm deg}\,g}}\B^g_{\balpha}\,\x^{\balpha},\qquad\forall\x\in\mathbb{R}^d.\end{equation}
Then $\M_n(g\,\bphi)=\sum_{\balpha\in\N^d_{2n+{\rm deg}g}}\phi_{\balpha}\,\B^g_{\balpha}$.
Importantly,
\begin{equation}
\label{eq:pos}
\M_n(g_j\,\bphi)\,\succeq\,0\quad\Leftrightarrow\quad\bphi(f^2\,g_j)\,\geq\,0\,,\quad\forall f\in\R[\x]_n\,.\end{equation}
Next, if $\bphi$ has a representing measure $\phi$ whose support is
contained in the set $\{\x:g(\x)\geq0\}$ then
$\M_n(g\,\bphi)\succeq0$ for all $n$ because $\langle\f,\M_n(g\,\bphi)\f\rangle=\int f^2\,gd\phi\geq0$,
for all $f\in\mathbb{R}[\x]_n$.

\paragraph{Homogenization}
Let $\bphi^n=(\phi_{\balpha})_{\balpha\in\N^d_{2n}}$. Its homogenization is the vector
$\tilde{\bphi}^n=(\tilde{\phi}^n_{i,\balpha})_{i+\balpha=2n}$ defined by:
\[\tilde{\phi}^n_{2n-\vert\balpha\vert,\balpha}=\phi^n_{\balpha},\quad \balpha\in\N^d_{2n},\]
and the homogenization $\widetilde{\M}_n(\tilde{\bphi}^n)$ of the moment matrix $\M_n(\bphi^n)$ 
is defined by:
\[\widetilde{\M}_n(\tilde{\bphi}^n)((i,\balpha),(j,\bbeta))\,=\,\M_n(\bphi^n)(\balpha,\bbeta)\,=\,
\tilde{\phi}^n_{i+j,\balpha+\bbeta},\quad i+\vert\balpha\vert=j+\vert\bbeta\vert=n.\]

For instance in dimension $d=2$, and for $n=1$:
\[\M_1(\bphi^n)=\left[\begin{array}{ccc}
\phi^n_{00} &\phi^n_{10} &\phi^n_{01}\\
\phi^n_{10} &\phi^n_{20} &\phi^n_{11}\\
\phi^n_{01} &\phi^n_{11} &\phi^n_{02}\end{array}\right]
\,=\,\widetilde{\M}_1(\tilde{\bphi}^n)=\left[\begin{array}{ccc}
\tilde{\phi}^n_{2,00} &\tilde{\phi}^n_{1,10} &\tilde{\phi}^n_{1,01}\\
\tilde{\phi}^n_{1,10} &\tilde{\phi}^n_{0,20} &\tilde{\phi}^n_{0,11}\\
\tilde{\phi}^n_{1,01} &\tilde{\phi}^n_{0,11} &\tilde{\phi}^n_{0,02}\end{array}\right]\]
{\color{blue}So the moment matrix $\widetilde{\M}_1(\tilde{\bphi}^n)$ of the homogenization 
$\tilde{\bphi}^n$ of $\bphi^n$, and the moment matrix $\M_n(\bphi^n)$ of $\bphi^n$,
are the same up to an obvious relabelling.}
\begin{theorem}(Blekherman \cite[Theorem 2.3]{Blek})
\label{th-Blek}
Let $\bphi^n$ and $\tilde{\bphi}^n$ be such that $\widetilde{\M}_n(\tilde{\bphi}^n)\succeq0$. 
If $s:={\rm rank}\,\widetilde{\M}_n(\tilde{\bphi}^n)\leq 3n-3$ when $n\geq 3$
(or $s\leq 5$ when $n=2$) then 
$\tilde{\bphi}^n$ has an atomic  representing measure on $\R^{d+1}$ supported on $s$ atoms.
\end{theorem}
For every $n$ and $\bphi^n=(\phi^n_{\balpha})_{\balpha\in\N^d_{2n}}$, define $\bphi^{\hat{n}}:=(\phi^n_{\balpha})_{\balpha\in\N^d_{2n-1}}$. That is, $\bphi^{\hat{n}}$ is the restriction
of the vector $\bphi^n$ to its ``moments" up to degree $2n-1$.
\begin{cor}(Fialkow \cite{Fialkow})
\label{cor-blek}
Let $\bphi^n$ and $\tilde{\bphi}^n$ be such that $\widetilde{\M}_n(\tilde{\bphi}^n)\succeq0$
and $s:={\rm rank}\,\widetilde{\M}_n(\tilde{\bphi}^n)\leq 3n-3$ (or $\leq 5$ if $n=2$). Then
$\bphi^{\hat{n}}\:(=(\phi_{\balpha})_{\vert\balpha\vert\leq 2n-1})$ has an atomic  
representing measure on $\R^d$.
\end{cor}
See also the comment after Theorem 1.3 in \cite[p. 948]{Fialkow-2}.
We will also need  the following consequence:
\begin{cor}
\label{precision}
Let $\bphi^n$ {\color{blue}(with $\bphi^n(1)>0$)} and $\tilde{\bphi}^n$ be such that $\widetilde{\M}_n(\tilde{\bphi}^n)\succeq0$
and $s={\rm rank}\,\widetilde{\M}_n(\tilde{\bphi}^n)\leq 3n-3$ (or $\leq 5$ if $n=2$). Then
$\bphi^{\hat{n}}\:(=(\phi_{\balpha})_{\vert\balpha\vert\leq 2n-1})$ 
has an atomic  representing measure on $\R^d$  supported on $r\,(\leq s)$ points,
and with mass $\phi^n(1)$.  If
$r=s$ then the whole sequence $\bphi^n$ has a representing measure supported on $s$ points of $\R^d$.
\end{cor}
\begin{proof}
 Let $\{(x_0(1),\x(1)),\ldots,(x_0(s),\x(s))\}\subset\R^{d+1}$,  be the support of
 $\tilde{\bphi}^n$ with associated (strictly) positive weights $\lambda_1,\ldots,\lambda_s$. 
 Let $\Delta:=\{i: x_0(i)\neq0\}$ {\color{blue}which is nonempty because
\[0\,<\phi^n(1)\,=\,\phi^n_{0\ldots 0}\,=\,\tilde{\phi}^n_{2n,0\ldots 0}\,=\,\tilde{\phi}^n(x_0^{2n})\,=\,\sum_{i=1}^s\lambda_i\, x_0(i)^{2n}\,.\]}
 Then for every $\balpha\in\N^d_{2n}$ with $\vert\balpha\vert<2n$,
 \begin{eqnarray*}
 \tilde{\phi}^n_{2n-\vert \balpha\vert,\balpha}\,=\,\phi^n_{\balpha}&=&
 \sum_{i=1}^s\lambda_i\,x_0(i)^{2n-\vert\balpha\vert}\prod_{j=1}^d x_j(i)^{\alpha_j}\\
 &=&\sum_{i\in\Delta}\lambda_i\,x_0(i)^{2n}\,\prod_{j=1}^d (\frac{x_j(i)}{x_0(i)})^{\alpha_j}
 \,=\,\int_{\R^d} \x^{\balpha}\,d\mu\,,
  \end{eqnarray*}
  with $\mu:=\sum_{i\in\Delta}\lambda_i\,x_0(i)^{2n}\delta_{(\frac{x_1(i)}{x_0(i)},\cdots,\frac{x_d(i)}{x_0(i)})}$. 
  That is, $\mu$ is a representing measure for $\bphi^{\hat{n}}$. In addition, $\mu(1)=\phi^n(1)$
  because $\phi^n(1)=\tilde{\phi}^n(x_0^{2n})=\mu(1)$.
  Next,
  for every $\balpha$ with $\vert\balpha\vert=2n$,
  \begin{eqnarray*}
 \tilde{\phi}^n_{0,\balpha}\,=\,\phi^n_{\balpha}&=&
 \sum_{i=1}^s\lambda_i\,x_0(i)^{2n-\vert\balpha\vert}\prod_{j=1}^d x_j(i)^{\alpha_j}\\
 &=&\sum_{i\in\Delta}\lambda_i\,x_0(i)^{2n}\,\prod_{j=1}^d (\frac{x_j(i)}{x_0(i)})^{\alpha_j}
 +\sum_{i\not\in\Delta}\lambda_i\,\prod_{j=1}^d x_j(i)^{\alpha_j}\\
 &=&{\color{blue}\int_{\R^d} \x^{\balpha}\,d\mu+\sum_{i\not\in\Delta}\lambda_i\,\prod_{j=1}^d x_j(i)^{\alpha_j}\,,}
  \end{eqnarray*}
 and therefore, if $\vert\Delta\vert=s$ we may and will conclude that
 $\phi^n_{\balpha}=\int_{\R^d}\x^{\balpha}\,d\mu$ for all $\balpha\in\N^d_{2n}$,
 i.e., $\mu$ is a representing measure of the whole sequence $\bphi^n$.
\end{proof}

\section{Main result}

With $\K$ as in \eqref{set-K}, let $v:=\max_{1\leq j\leq m}\lceil\mathrm{deg}(g_j)/2\rceil$.
Given a sequence $\bphi^n=(\phi^n_{\balpha})_{\balpha\in\N^d_{2n}}$, recall the notation
$\bphi^{\hat{n}}=(\phi^n_{\balpha})_{\balpha\in\N^d_{2n-1}}$, i.e.,
$\bphi^{\hat{n}}$ is the restriction of $\bphi^n$ to moments up to degree $2n-1$.

Our first result below {\color{blue}(and announced as Theorem \ref{intro-th1} in \S \ref{sec:intro})}
provides a sufficient condition to solve the truncated $\K$-moment problem. Recall that $v=\max_j d_j$.

\begin{theorem}
\label{th1}
With $n\geq \max[2,v]$, let $\bphi^n=(\phi_{\balpha})_{\balpha\in\N^d_{2n}}$ be such that 
\begin{equation}
\label{th1-1}
\phi(1)\,=\,1\,;\quad \M_n(\bphi^n)\,\succeq0\,;\quad \M_{n-d_j}(g_j\,\bphi^n)\,\succeq0\,,\quad j=1,\ldots,m\,.
\end{equation}
If $s:=\mathrm{rank}(\M_n(\bphi^n))\,\leq\,n-v+1$ 
then the truncated sequence $\bphi^{\hat{n}}=(\phi^n_{\balpha})_{\balpha\in\N^d_{2n-1}}$ has a representing measure supported on 
at most $r\leq s$ points of $\K$. In addition, if $r=s$ then
the whole sequence $\bphi^{n}=(\phi^n_{\balpha})_{\balpha\in\N^d_{2n}}$ has a representing measure supported on $s$ points of $\K$.
\end{theorem}
The case $n=1$ is only meaningful when $v=1$. In this case
$s\leq 1$ implies that the whole sequence $\phi^n$ has a representing measure,
the Dirac at some point $\x\in\R^d$.\\

{\color{blue}\subsection*{Brief sketch of the proof}While  a detailed
proof is postponed to \S \ref{appendix-1}, we briefly describe its main steps.
We consider the two cases $r=s$ and $r<s$ but the machinery is essentially the same.
For the case $r=s$, one first uses Corollary \ref{cor-blek} to observe that $\bphi^{n}$ has an atomic 
representing measure supported on $s$ points $\x(1),\ldots,\x(s)\in \R^d$. Next we consider $j\leq m$, fixed, arbitrary,
and let $\Gamma:=\{i: g_j(\x(i))\geq0\}$. By the localizing constraint $\M_n(g_j\,\bphi^n)\succeq0$, 
$\vert\Gamma\vert\geq 1$. Then using interpolating polynomials of sufficient degree,
we prove by contradiction that $\vert\Gamma\vert=s$, i.e., $g_j(\x(i))\geq0$, $i=1,\ldots,s$, and as $j$ was arbitrary,
$\bphi^n$ has a representing measure on $\K$. The rank condition plays a crucial role 
for the existence of such interpolating polynomials). The case $r<s$ follows the same steps.}

\begin{remark}
\label{rem-blek}
{\color{blue}The rank condition in Theorem \ref{th1} (or in Theorem \ref{th-main}) is more 
restrictive than in Theorem \ref{th-Blek}.
However this is quite natural as the rank condition must also guarantee an important additional
feature of a representing measure of $\bphi^{\hat{n}}$, namely its support
should be contained in $\K$, whereas in Corollary \ref{cor-blek},
$\bphi^{\hat{n}}$ being  in no way related to any set $\K\subset\R^d$,
one cannot expect any localization property of  its representing measure whenever
the latter exists.}
\end{remark}
\begin{example}
{\color{blue}Consider the univariate atomic measure $\phi$ supported on the $3$ points
$\{1, 2,  3\}\subset\R$ with respective weights $\{1/3,1/3,1/3\}$, that is,
\begin{equation}
\label{eq:ex}
\phi=\frac{1}{3}(\delta_{\{1\}}+\delta_{\{2\}}+\,\delta_{\{3\}})\end{equation}
and let $\bphi^3=(\phi_{\balpha})_{\balpha\in\N_6}$. By construction $\M_3(\bphi^3)\succeq0$, with $\mathrm{rank}(\M_3(\bphi^3))=3\leq 3\cdot 3-3=6$.
Therefore by Corollary \ref{cor-blek}, $\bphi^{\hat{3}}$ has a representing measure on $\R$ (and $\phi$ in \eqref{eq:ex}
is one such measure which in fact represents the whole sequence $\bphi^3$
and not only $\bphi^{\hat{3}}$). But with $x\mapsto g(x):=x(1-x)$,
one may check that the localizing matrix $\M_2(g\,\bphi^3)$ is not positive semidefinite $(\M_2(g\,\bphi^3)\not\succeq0$);
for instance, the first diagonal element is negative. Hence we may conclude that
$\bphi^3$ cannot have a representing measure on $[0,1]$.}
\end{example}

\begin{example}
{\color{blue}The following univariate example shows a degree-$2$ moment matrix $\M_2(\bphi^2)\succeq0$ with
$\mathrm{rank}(\M_2(\bphi^2))=3\,(\leq5)$, and $\M_1(g\,\bphi^2)\succeq0$ with $g(x)=x(1-x)$. 
So even though Blekherman's
condition holds and the localizing condition holds, one exhibits
a $3$-atomic measure with $2$ atoms in $[0,1]$ and $1$ atom outside.  Hence
there is a representing measure of $\bphi^2$ which is not supported on $[0,1]$. Namely,
\[\phi=0.495\,\delta_{\{0.24\}}+0.495\,\delta_{\{0.26\}}+0.01\,\delta_{\{-0.01\}}\]
\[\M_1(g\,\bphi^2)\,=\,\left[\begin{array}{cc}0.185425 &0.04592721\\0.04592721 &0.01163320\end{array}\right]\,\succeq\,0\,.\]
}
\end{example}
\begin{remark}
{\color{blue}
As noted in Remark \ref{rem-blek}, the rank condition in Theorem \ref{th1} is quite strong, especially when compared to Blekherman's condition in Theorem \ref{th-Blek} and Corollary \ref{cor-blek}. In the proof of Theorem \ref{th1}, 
this condition is a guarantee that
$n-v$ is a sufficiently large interpolating degree 
to ensure that for every $j$, $g_j(\x)\geq0$ for all points $\x$ in the finite support $X$ of the representing measure. 

Notice that the kernel of $\M_n(\bphi^n)$ provides some information on the set $X$. Indeed as soon as Blekherman's condition $\mathrm{rank}(\M_n(\bphi^n))\leq 3n-3$ holds,  each vector of $\mathrm{ker}(\M_n(\bphi^n))$ 
is the coefficient vector  of a polynomial $q$ that vanishes on $X$ (if $\mathrm{deg}(q)\leq n-1$ whenever $r<s$).
If the information on $X$ provided by such polynomials $q$ yields a smaller interpolating degree (e.g., by invoking results in \cite{Velasco}) then the rank condition in Theorem 6 and 7 could be weakened. However, this would require assumptions on the polynomials $q\in\R[\x]_{n-1}$ of $\mathrm{ker}(\M_n(\bphi^n))$ while the present condition considers the worst case scenario with no \`a priori knowledge on $X$.}
\end{remark}
We next investigate the consequence of Theorem \ref{th1} for polynomial optimization
{\color{blue}in our second result (announced as Theorem \ref{intro-th-main} in \S \ref{sec:intro}).}
With $f$ as in \eqref{def-pb-1}, define 
$d_f:=\lceil\mathrm{deg}(f)/2\rceil$ and  recall that
$d_j:=\lceil\mathrm{deg}(g_j)/2\rceil$, for all $j=1,\ldots,m$,
and $v:=\max_{j\leq m} d_j$.  We also assume the $\mathrm{deg}(f)\geq2$.\\

\begin{theorem}
\label{th-main}
Let $\P$ be as in \eqref{def-pb-1}.
For every $n$ with $n\geq v$ and $2n-1\geq \mathrm{deg}(f)$, let 
$\bphi^n=(\phi^n_{\balpha})_{\balpha\in\N^d_{2n}}$ 
be  an optimal solution of the semidefinite relaxation $\Q_n$:

(i) If $s={\rm rank}(\M_n(\bphi^n))\leq n-v+1$ then
$\phi^n(f)=f^*$, that is, $\rho_n=f^*$ and the relaxation $\Q_n$ is exact. Moreover, $\bphi^{\hat{n}}$
has a representing measure supported on at most $s$ global minimizers of $\P$.

(ii) Next assume that $\mathrm{deg}(f)\leq 2$ and $\mathrm{deg}(g_j\leq 2)$ for all $j$,
so that $v=1$ (and $\P$ is a QCQP). Then one may exhibit a probability measure
supported on global minimizers of $\P$.
\end{theorem}
A detailed proof is postponed to \S \ref{appendix-2}. 
So at an optimal solution $\bphi^n$ of $\Q_n$,
Theorem \ref{th-main} provide  a sufficient rank condition on 
$\M_n(\bphi^n)$ to ensure that the semidefinite relaxation $\Q_n$ is exact. In addition,
for QCQPs one may extract global minimizers by
looking at moment matrices (submatrices of the moment matrix $\M_n(\bphi^n)$),
 of degree lower than $n$.
\vspace{.2cm}

\section{The unconstrained case}
\label{sec:exactness}
In \cite[p. 72]{Blek} the author claims:
 \emph{``Theorem 2.1 also leads to an interesting stopping criterion for sum of squares
relaxations. Sum of squares methods lead to a hierarchy of relaxations indexed by
degree. $\ldots$ {\color{blue}One of the important questions in this area is
as follows: when can we guarantee that we obtained the actual optimum, and thus
stop computing relaxations of higher degree?}}, and  later still in p. 72:

\noindent
``{\bf Stopping criterion for  SOS relaxations.}\\
\emph{Suppose that the sum of squares relaxation truncated in degree $2d$ with $d\geq3$,
returns an optimal linear functional with  moment matrix of rank at most $3d-3$. Then
the relaxation is exact."}  

{\color{blue}The above quoted statements from \cite{Blek} are misleading the reader because} for constrained POPs in \eqref{def-pb-1}, 
even if in an optimal solution $\bphi^n$ of the relaxation $\Q_n$, 
$\mathrm{rank}(\M_n(\bphi^n))\leq 3n-3$ (in \cite{Blek}
$d$ is the degree whereas for us $n$ is the degree)  
then  a representing measure for $\bphi^n$ is \emph{not} guaranteed to be supported on $\K$.
Precisely, Theorem \ref{th-main} provides a more restrictive rank condition to ensure that 
{\color{blue}the support of any representing measure of (the restricted) $\bphi^{\hat{n}}$
is indeed contained in $\K$}; see Remark \ref{rem-blek}.
{\color{blue}However on the other hand}, we next show that indeed Blekherman's result is useful for 
\emph{unconstrained} polynomial optimization, that is, so solve:
\begin{equation}
\label{def:uncons}
\P:\quad f^*\,=\,\inf_{\x\in\R^d} f(\x)
\end{equation}
where $f\in\R[\x]_{2n}$ with $n\geq2$ (the case $n=1$ being trivial). 
With $\P$ is associated the \emph{single} semidefinite relaxation:
\begin{equation}
 \label{eq:uncons-primal}
 \rho\,=\,\inf_{\bphi\in\R^{s(n)}} \{\,\phi(f):\: \phi(1)\,=\,1\,;\: \M_n(\bphi)\,\succeq\,0\,\}\,.
\end{equation}
Indeed there is \emph{no} hierarchy to consider. Either $f-f^*$ is SOS and $\rho=f^*$, or
$f-f^*$ is not an SOS and then $\rho<f^*$ (with possibly $\rho=-\infty$). The dual of \eqref{eq:uncons-primal} reads:
\begin{equation}
 \label{eq:uncons-dual}
 \rho^*\,=\,\sup_{\lambda} \{\,\lambda:\: f-\lambda \in\Sigma[\x]_n\,\}\,,
\end{equation}
where $\Sigma[\x]_n\subset\R[\x]_{2n}$ is the convex cone of SOS polynomials of degree at most $2n$.
\begin{theorem}
Consider the unconstrained POP in \eqref{def:uncons} and its associated (single) semidefinite relaxation
\eqref{eq:uncons-primal}. Let $\rho>-\infty$  and let $\bphi^*$ be an optimal solution of \eqref{eq:uncons-primal}. Suppose
that $r:=\mathrm{rank}(\M_n(\bphi^*))\leq 3n-3$ if $n\geq3$ or $r=:\mathrm{rank}(\bphi^*)\leq 5$ if $n=2$. 
Then $\rho=f^*$ and there exist some $k\,(\leq r)$ global minimizers.
\end{theorem}
\begin{proof}
We prove the result when $n\geq3$ while the arguments are similar for the case $n=2$.
Let $\tilde{f}$ (or $\mathrm{hom}(f)$) be the homogenization of $f$, that is, $\tilde{f}\in\R[x_0,\x]_{2n}$ with
 \begin{equation}
 \label{homog}
 \tilde{f}(x_0,\x)\,=\,\left\{\begin{array}{rl} x_0^{2n}\,f(\x/x_0)&\mbox{if $x_0\neq0$}\\
 f_{2n}(\x)&\mbox{if $x_0=0$}\end{array}\right.\,,\end{equation}
 where $f_{2n}$ is the degree-$2n$ homogeneous part of $f$. 
 So $\tilde{f}(1,\x)=f(\x)$ and $\tilde{f}(0,\x)=f_{2n}(\x)$ for all $\x\in\R^d$.
  In particular
if a polynomial is nonnegative or SOS then so is its homogenization and the converse is true as well. 

Slater's condition holds for \eqref{eq:uncons-primal}. Indeed let $\mu$ be the 
Gaussian measure $\mathcal{N}(0,\mathbf{I})$ on $\R^d$ (with identity matrix $\mathbf{I}$ as covariance matrix). Then $\M_n(\mu)\succ0$ and $\mu(f)$ is finite.
Therefore $\rho=\rho^*$. Next, by optimality of $\bphi^*$ and the necessary KKT optimality conditions,
there exists  $\X^*\succeq0$ such that
\[f(\x)-\rho\,=\,\v_n(\x)^T\,\X^*\,\v_n(\x)\,\quad\forall\x\in\R^d\,;\quad \langle\M_n(\bphi^*),\X^*\rangle\,=\,0\,.\]
Recall the homogenization  $\tilde{\bphi}^*=(\tilde{\phi}^*_{2n-\vert\balpha\vert,\balpha})$,
$\balpha\in\N^d_{2n}$ of $\bphi$ which reads
$\tilde{\phi}^*_{2n-\vert\balpha\vert,\balpha}=\phi^*_{\balpha}$ for every $\balpha\in\N^d_{2n}$.
{\color{blue}Therefore
\begin{equation}
\label{aa}
\tilde{\bphi}^*(\tilde{p})\,=\,\bphi^*(p)\,,\quad \forall p\in\R[\x]_{2n}\,.\end{equation}
In particular, letting $\w_n(x_0,\x):=\mathrm{hom}(\v_n(\x))=(x_0^{2n-\vert\balpha\vert}\x^{\balpha})_{\balpha\in\N^d_{2n}}$, observe that
\[f(\x)-\rho\,=\,\langle\v_n(\x)\,\v_n(\x)^T,\X^*\rangle\,\:\Rightarrow 
\mathrm{hom}(f-\rho)\,=\,\langle\w_n(x_0,\x)\,\w_n(x_0,\x)^T,\X^*\rangle\,,\]
for all $(x_0,\x)\in\R^{d+1}$, and therefore, the homogeneous moment matrix 
$\widetilde{\M}_n(\tilde{\bphi}^*)\,(=\M_n(\bphi^*)$) of $\tilde{\bphi}^*$ satisfies
\begin{eqnarray}
\nonumber
0\,=\,\langle\X^*,\M_n(\bphi^*)\rangle&=&\bphi^*(f-\rho)\\
\nonumber
&=&
\tilde{\bphi}^*(\mathrm{hom}(f-\rho))\quad\mbox{[by \eqref{aa}]}\\
\nonumber
&=&\tilde{\bphi}^*\left(\langle\X^*,\w_n(x_0,\x)\,\w_n(x_0,\x)^T\rangle\right)\\
\label{aaa}
&=&\langle\X^*,\widetilde{\M}(\tilde{\bphi}^*)\rangle\,.
\end{eqnarray}
Next, if $r:=\mathrm{rank}(\M_n(\bphi^*))\leq3n-3$ then by Theorem \ref{th-Blek},
$\tilde{\bphi}^*$ has a representing measure $\tilde{\phi}^*$ on $\R^{d+1}$ supported on  $r$ atoms
$\{(x_0(1),\x(1)),\ldots(x_0(r),\x(r))\}\subset\R^{d+1}$, and
\[\phi^*_{\balpha}\,=\,\int_{\R^{d+1}}x_0^{2n-\vert\balpha\vert}\x^{\balpha}\,d\tilde{\phi}^*(x_0,\x)
\,=\,\sum_{j=1}^r\lambda_j\,(x_0(j)^{2n-\vert\balpha\vert}\x(j)^{\balpha}\,,
\quad\forall\balpha\in\N^d_{2n}\,.\]
In particular, as
\[(x_0,\x)\mapsto \mathrm{hom}(f-\rho)\,=\,\tilde{f}(\x)-\rho\,x_0^{2n}\,,\]
then by using \eqref{aaa}, one obtains
\begin{eqnarray*}
0\,=\,\bphi^*(f)-\rho&=&\bphi^*(f-\rho)\\
&=&\tilde{\bphi}^*(\mathrm{hom}(f-\rho))\,=\,
\int_{\R^{d+1}} \underbrace{\tilde{f}(x_0,\x)-\rho\,x_0^{2n})}_{SOS}\,d\tilde{\phi}^*(x_0,\x)\,,\end{eqnarray*}
and therefore $\tilde{f}(x_0(j),\x(j))-\rho\,x_0(j)^{2n}=0$ for all $j=1,\ldots,r$.
Next, let $\Gamma:=\{\,j: \:x_0(j)\neq0\,\}$, and assume that $\Gamma\neq\emptyset$.}
Then invoking \eqref{homog}, one obtains
\[x_0(j)^{2n}\,(\,f(\x(j)/x_0(j))-\rho)\,=\,0\,,\quad\forall j\,\in\,\Gamma\,,\]
which implies $f(\x(j)/x_0(j))=\rho$ for all $j\in\Gamma$, and therefore for every $j\in\Gamma$,
$\x(j)/x_0(j)\in\R^d$ is a global minimizer of $f$ and $f^*=\rho$. It remains to prove that 
$\Gamma\neq\emptyset$. But this follows from $1=\bphi^*(1)\,=\,\int x_0^{2n}d\tilde{\phi}^*$.
\end{proof}
In the atomic support of $\tilde{\bphi}^*$, the above proof needs to 
treat separately points with $x_0(j)\neq0$ from points with $x_0(j)=0$. Indeed
if $\tilde{\bphi}^*$ has a representing measure $\tilde{\phi}^*$
on $\R^{d+1}$, only the subsequence 
$(\phi^*_{\balpha})_{\vert\balpha\vert\leq 2n-1}$ has a representing measure $\phi^*$ on $\R^d$. Then 
\[0\,=\,\int_{\R^{d+1}}\tilde{f}(x_0,\x)-\rho\,x_0^{2n})\,d\tilde{\phi}^*
(x_0,\x)\quad\not\Rightarrow 0\,=\,\int_{\R^d}(f-\rho)\,d\phi^*(\x)\,,\]
because as $\mathrm{deg}(f)=2n>2n-1$, $\bphi^*(f-\rho)\neq\int_{\R^d}(f-\rho)\,d\phi^*$ in general. However
$\int_{\R^d} (f-\rho)\,d\mu=0$ for the measure $\mu:=\sum_{j\in \Gamma}\lambda_j\,\delta_{\{\x(j)/x_0(j)\}}$
on $\R^d$.

\section{Conclusion}

We have provided a rank condition on the moment matrix 
of an optimal solution of the degree-$n$ semidefinite relaxation
of the Moment Hierarchy applied to POP. When satisfied, the
corresponding semidefinite relaxation is exact, i.e.,
the Moment-SOS hierarchy has finite converge
(and for QCQPs, global minimizers can be extracted). These conditions
are in the spirit of Blekherman's  condition,
i.e., are concerned with a single moment matrix in contrast to 
Curto \& Fialkow's flat extension condition. While they are
are more restrictive, they apply to constrained POPs whereas 
Blekherman's condition only helps  for unconstrained POPs.

\section*{Acknowledgement} The author is grateful to an anonymous referee 
for providing the reference \cite{Velasco}.
\backmatter

\section*{Declarations}

\begin{itemize}
\item Funding:
The author is supported by the ANITI AI Cluster program under the Grant agreement ANR-23-IACL-0002, and also by the Marie-Sklodovska-Curie european doctoral network TENORS, under grant 101120296.
This research is also part of the programme DesCartes and is supported by the National Research Foundation, Prime Minister's Office, Singapore under its Campus for Research Excellence and Technological Enterprise (CREATE) programme
\item Conflict of interest/Competing interests (check journal-specific guidelines for which heading to use): Not applicable
\item Ethics approval and consent to participate: Not applicable
\item Consent for publication: Not applicable
\item Data availability: Not applicable
\item Materials availability: Not applicable
\item Code availability: Not applicable
\item Author contribution: Not applicable
\end{itemize}

\begin{appendices}
\section{Proof of Theorem \ref{th1}}
\label{appendix-1}
\begin{proof}
Observe that if $s\leq\,n-v+1$ then $s\leq 3n-3$ if $n\geq3$ and $s\leq 5$ if $n=2$.
Then by Corollary \ref{cor-blek}, $\bphi^{\hat{n}}$ has an atomic  representing measure 
$\phi$ supported on at most $r$ points $\x(1),\ldots \x(r)\in\R^d$ with $r\leq s$. That is, there exist
$\lambda_i>0$, $i=1,\ldots,r$, such that
\[\phi\,=\,\sum_{i=1}^r\lambda_i\,\delta_{\x(i)} \quad\mbox{and}\quad\int \x^{\balpha}\,
d\phi\,=\,\phi^n_{\balpha}\,,\:\forall\balpha\in\N^d_{2n-1}\,.\]
Let $1\leq j\leq m$ be fixed arbitrary. 

\noindent
{\bf Case $r=s$.} Then by Corollary \ref{precision}, the \emph{whole} sequence $\bphi^n$ 
has a representing measure 
$\phi$ supported on $s$ points $\x(1),\ldots,\x(s)\in \R^d$.
Next, since $2(n- d_j)+\mathrm{deg}(g_j)\leq 2n$,
the localizing matrix $\M_{n-d_j}(g_j\,\bphi)$ contains only moments of degree at most
$2n$ and therefore, $\phi^n(q\,g_j)=\int q\,g_j\,d\phi$ for all $q\in\R[\x]_{2(n-d_j)}$. Next,
by \eqref{eq:pos},
\begin{equation}
\label{posi}
0\,\preceq\,\M_{n-d_j}(g_j\,\bphi^n)\,\Rightarrow \phi^n(q\,g_j)=\int q\,g_j\,d\phi\,\geq\,0\,,\end{equation}
for every SOS polynomial $q\in\R[\x]_{2(n-d_j)}$.
From this we deduce that at least one point $\x(i)$ satisfies $g_j(\x(i))\geq0$, i.e.,
the cardinality $\vert\Gamma\vert$ of the set 
$\Gamma:=\{i: g_j(\x(i))\geq0\}$ is at least $1$. Next, suppose that $\vert\Gamma\vert<s$. Then
consider the polynomial $p\in\R[\x]_{2\vert\Gamma\vert}$ defined by:
\begin{equation}
\label{poly-p}
\x\mapsto p(\x)\,:=\,\prod_{i\in\Gamma}\left(\sum_{k=1}^d(x_k-x_k(i))^2\right)\,=\,
\prod_{i\in\Gamma}\Vert \x-\x(i)\Vert^2\,.\end{equation}
Then $p(\x(i))=0$ for every $i\in\Gamma$ and $p(\x(i))\,>\,0$ for every $i\not\in\Gamma$.
 Moreover, $p$ is an SOS of degree $2\vert\Gamma\vert<2s$, and 
 as $s\leq n-v+1$,
 $2\vert\Gamma\vert\leq 2n-2v\leq 2(n-d_j)$.
 Hence  by \eqref{posi}, $0\leq\phi^n(p\,g_j)=\int p\,g_j\,d\phi$, 
and one obtains the contradiction
\[0\,\leq\,\int p\,g_j\,d\phi\\
\,=\,\sum_{i\in\Gamma}\lambda_i\,\underbrace{p(\x(i))}_{=0} \,g_j(\x(i))+\sum_{i\not\in\Gamma}\underbrace{\lambda_i\,p(\x(i))}_{>0} \,\underbrace{g_j(\x(i))}_{<0}\,.\]
Therefore $\vert\Gamma\vert=s$ which implies $g_j(\x(i))\geq0$ for all $i=1,\ldots,s$.
As $j$ was arbitrary, $\x(i)\in\K$ for all $i=1,\ldots,s$, i.e., $\phi$ is supported on $\K$.\\

\noindent
{\bf Case $r<s$.} Then by Corollary \ref{precision}, $\bphi^{\hat{n}}$ has a representing measure 
$\phi$ supported on $r$ points $\x(1),\ldots,\x(r)\in\R^d$.
Next, since $2(n-1-d_j)+\mathrm{deg}(g_j)\leq 2n-1$,
the localizing matrix $\M_{n-1-d_j}(g_j\,\bphi)$ contains only moments of degree at most
$2n-1$, and therefore $\phi^{n}(q\,g_j)=\phi^{\hat{n}}(q\,g_j)=\int q\,g_j\,d\phi$
for all $q\in\R[\x]_{2(n-1-d_j)}$.
Moreover $\M_{n-d_j}(g_j\,\bphi)\succeq0\Rightarrow \M_{n-1-d_j}(g_j\,\bphi)\succeq0$,
 and therefore, by \eqref{eq:pos},
\begin{equation}
\label{posi-r}
0\,\preceq\,\M_{n-1-d_j}(g_j\,\bphi^n)\,\Rightarrow \phi^n(q\,g_j)\,=\,
\phi^{\hat{n}}(q\,g_j)\,=\,\int q\,g_j\,d\phi\,\geq\,0\,,\end{equation}
for every SOS polynomial $q\in\R[\x]_{2(n-d_j-1)}$. Again we deduce that at least one point $\x(i)$ satisfies $g_j(\x(i))\geq0$. So the cardinality of the set 
$\Gamma:=\{i: g_j(\x(i))\geq0\}$ is at least $1$. Next, suppose that $\vert\Gamma\vert<r$. 
Let $p\in\R[\x]_{2\vert\Gamma\vert}$ be the SOS polynomial defined in \eqref{poly-p}.
As $r<n-v+1$ and $\vert\Gamma\vert<r$, $\vert\Gamma\vert\leq n-v-1$,  and so
$p$ is an SOS of degree at most $2(n-d_j-1)$. Moreover, as
\[\mathrm{deg}(g_j\,p)\,\leq\,2(n-d_j-1)+\mathrm{deg}(g_j)\leq 2n-2\,,\]
$\phi^n(p\,g_j)\,=\,\phi^{\hat{n}}(p\,g_j)=\int p\,g_j\,d\phi$. By \eqref{posi-r},
one obtains the contradiction
\begin{eqnarray*}
0\,\leq\,\phi^{\hat{n}}(p\,g_j)&=&\int p\,g_j\,d\phi\\
&=&\sum_{i\in\Gamma}\lambda_i\,\underbrace{p(\x(i))}_{=0} \,g_j(\x(i))+\sum_{i\not\in\Gamma}\underbrace{\lambda_i\,p(\x(i))}_{>0} \,\underbrace{g_j(\x(i))}_{<0}\,.
\end{eqnarray*}
Therefore $\vert\Gamma\vert=r$, which implies $g_j(\x(i))\geq0$ for all $i=1,\ldots,r$,
and as $j$ was arbitrary, $\x(i)\in\K$ for all $i=1,\ldots,r$, i.e., 
$\phi$ is supported on $\K$.
\end{proof}

 \section{Proof of Theorem \ref{th-main}}
 \label{appendix-2}
 \begin{proof}
 (i) As $\mathrm{deg}(f)\geq2$ and $2n-1\geq\mathrm{deg}(f)$, then $n\geq2$.
 By Theorem \ref{th1}, $\bphi^{\hat{n}}$ has an atomic representing measure 
$\phi$ supported on at most $r\,(\leq s)$ points $\x(1),\ldots,\x(r)\in\K$. Moreover, as $\mathrm{deg}(f)\leq 2n-1$, 
\[f^*\,\geq\,\rho_n\,=\,\phi^n(f)\,=\,\phi^{\hat{n}}(f)\,=\,\int_{\K} f\,d\phi\,=\,\sum_{i=1}^r
\underbrace{\lambda_i}_{\in (0,1]}\,\underbrace{f(\x(i))}_{\geq f^*}\,\geq f^*\,,\]
(with $\sum_i\lambda_i=1$) which proves that $\Q_n$ is exact, and $f(\x(i))=f^*$ for all $i=1,\ldots,r$.

(ii)  For every $k\leq n$, let $\bmu^{k}:=(\phi^n_{\balpha})_{\vert\balpha\vert\leq 2k}$.
Suppose that $r={\rm rank}(\M_{n-1}(\bmu^{n-1}))={\rm rank}(\M_n(\bphi^n))$. 
As $v=1$ and $\M_{n-1}(g_j\,\bphi^n)\succeq0$ for all $j=1,\ldots,m$,
then by the flat extension theorem of Curto and Fialkow \cite{Curto,Fialkow},
$\bphi^n$ has an atomic representing measure $\mu$ supported on 
$r$ atoms of $\K$, and the $r$ atoms can be extracted; see e.g. \cite{extract}. hence the result follows.

On the other hand, if
${\rm rank}(\M_{n-1}(\bmu^{n-1}))<{\rm rank}(\M_n(\bphi^n))\,(\leq n)$, then as $n\geq2$,
\begin{equation}
\label{iter}
{\rm rank}(\M_{n-1}(\bmu^{n-1}))\leq n-1 =(n-1 -v+1)
\quad\mbox{and}\quad\mu^{n-1}(f)=\phi^n(f)\,=\,f^*\,.\end{equation}
Observe that $\bmu^{n-1}$ satisfies the condition in Theorem \ref{th1} (with $n-1$ in lieu of $n$).

-- If $n-1=1$ then ${\rm rank}(\M_{n-1}(\bmu^{n-1}))\leq 1$ and in fact $=1$
(because $\mu^{n-1}(1)=1$), which in turn implies that
$\bmu^{n-1}$ has a representing measure  on $\K$ which is the Dirac at some point $\x^*\in\K$,
and the result follows.

-- If $n-1>1$, then $2n-2>2$ and so $\mu^{n-2}(f)=\phi^n(f)=f^*$. 
Moreover as ${\rm rank}(\M_{n-1}(\bmu^{n-1}))\leq n-1 =(n-1 -v+1)$,
then by Theorem \ref{th1}, 
$\bmu^{\widehat{n-1}}$ has a representing measure  on $\K$ supported on $r\leq n-1$ points of $\K$
and as $\mu^{\widehat{n-1}}(f)=\phi^n(f)=f^*$ (because $2n-3\geq 2$), the $r$ points are global minimizers of $\P$.

We may and will repeat the argument with $\bmu^{n-1}$ in lieu of $\bphi^n$. Suppose that 
${\rm rank}(\M_{n-2}(\bmu^{n-2}))={\rm rank}(\M_{n-1}(\bmu^{n-1}))$. Then 
as $\M_{n-2}(g_j\,\bmu^{n-1})\succeq0$ for all $j$,
again by the flat extension theorem of Curto \& Fialkow,
$\bmu^{n-1}$ has an atomic representing measure supported
on $r={\rm rank}(\M_{n-1}(\bmu^{n-1}))$ ($\leq n-1$) points of $\K$
that can be extracted; see \cite{extract}. As  
$\mu^{n-1}(f)=\phi^n(f)=f^*$ and $f\geq f^*$ on $\K$, then all points are global minimizers
and the result follows.

On the other hand, if ${\rm rank}(\M_{n-2}(\bmu^{n-2}))<{\rm rank}(\M_{n-1}(\bmu^{n-1})) \,(\leq n-1)$
then 
\[{\rm rank}(\M_{n-2}(\bmu^{n-2}))\leq n-2
\quad\mbox{and}\quad \mu^{n-2}(f)=\phi^n(f)=f^*\,,\]
because $n-2\geq 1$. That is,  we are back to the case \eqref{iter} but now with $n-2$ instead of $n-1$.
By iterating, we  stop with a measure $\bmu^{n-k}$ and either:
\begin{itemize}
\item $n-k>1$, ${\rm rank}(\M_{n-k}(\bmu^{n-k}))={\rm rank}(\M_{n-k-1}(\bmu^{n-k-1}))$, and in addition,
$\mu^{n-k}(f)=\phi^n(f)=f^*$, $\M_{n-k-1}(g_j\,\bmu^{n-k})\succeq0$, $j=1,\ldots,m$. Hence
by Flat extension theorem of Curto \& Fialkow, $\bmu^{n-k}$ has a representing measure supported on
${\rm rank}(\M_{n-k}(\bmu^{n-k}))$ global minimizers of $\P$,
and the minimizers can be extracted \cite{extract}; hence the result follows.
\item or $n-k=1$ in which case ${\rm rank}(\M_{n-k}(\bmu^{n-k}))=1$
because $\mu^{n-k}(1)=\phi^n(1)=1$. Then $\bmu^{n-k}$ is represented by the Dirac  measure
at the point $\x^*=(\mu^{n-k}(x_1),\ldots,\mu^{n-k}(x_d))$, and as $v=1$,
$\M_{n-k-1}(g_j\,\bmu^{n-k})=g_j(\x^*)\geq 0$ for every $j=1,\ldots,m$. Hence $\x^*\in\K$,
$\mu^{n-k}(f)=\phi^n(f)=f^*$, and so the result follows.
\end{itemize}
\end{proof}
\end{appendices}

\bibliography{lasserrebib1}
\end{document}